\documentclass[11pt]{article}                 

\usepackage{amssymb}
\usepackage{amsfonts} 
\usepackage{amsmath,amsthm}  
\usepackage{stmaryrd}
\usepackage{latexsym}
\usepackage{graphicx}          \DeclareGraphicsExtensions{.eps}
\usepackage{a4wide}
\usepackage{url}
\usepackage[all]{xy}
\usepackage{enumerate}
\usepackage[T1]{fontenc}

\usepackage{times}

\newtheorem{theorem}{Theorem}[section]   
\newtheorem{corollary}[theorem]{Corollary}  
\newtheorem{lemma}[theorem]{Lemma}

\newtheorem{example}[theorem]{Example}
\newtheorem{remark}[theorem]{Remark}

\numberwithin{equation}{section}         

\def\en{\enspace}

\def\ab#1{|#1|}

\def\c #1{{\cal #1}}

\def\C{\mathbf{C}}
\def\I{\mathbf{I}}

\def\a{\alpha}
\def\b{\beta}

\def\g{\gamma}
\def\G{\varGamma}
\def\ph{\varphi}

\def\bo{\Box}
\def\di{\Diamond}
\def\dit{{\langle t \rangle}}

\def\imp{\Rightarrow}
\def\liff{\Leftrightarrow}

\def\join{{\textstyle\bigvee}}
\def\meet{{\textstyle\bigwedge}}

\def\sub{\subseteq}

\def\int{\mathop{\rm int}} 

\def\vec#1#2{#1_1,\ldots{},#1_{#2}}
\def\Vec#1#2{#1_0,\ldots{},#1_{#2}}

\title{Tangled Closure Algebras}
\author{Robert Goldblatt\thanks{School of Mathematics and Statistics, Victoria University of Wellington, New Zealand.
{\tt sms.vuw.ac.nz/\~{}rob/}}
\enspace and\enspace
 Ian Hodkinson\thanks{Department of Computing, Imperial College London, UK.
{\tt www.doc.ic.ac.uk/\~{}imh/}.
}}

\begin{document}
\maketitle

\centerline{Dedicated to Bernhard Banaschewski on the occasion of his 90th birthday}

\bigskip
\begin{abstract}
The tangled closure of a collection of subsets of a topological space is the largest subset in which each member of the collection is dense. This operation models a logical  `tangle modality' connective, of significance in finite model theory.
 Here we study an abstract equational algebraic formulation of the operation which generalises the McKinsey-Tarski theory of closure algebras. We show that any dissectable tangled closure algebra, such as the algebra of subsets of any metric space without isolated points,  contains copies of every finite  tangled closure algebra. We then exhibit an example of a tangled closure algebra that cannot be embedded into any complete tangled closure algebra, so it has no MacNeille completion and no spatial representation.
\end{abstract}

\emph{Keywords}: closure algebra, tangled closure, tangle modality,  fixed point, quasi-order, Alexandroff topology,  dense-in-itself, dissectable, MacNeille completion.

\section{Introduction}

McKinsey  and Tarski \cite{mcki:alge44,mcki:clos46} defined a \emph{closure algebra} as a Boolean algebra equipped with a unary function $\C$ that satisfies  axioms of Kuratowski \cite{kura:oper22} for the operation of forming the topological closure of a set. They graphically revealed  the intricacy of the structure of many familiar topological spaces by defining a notion of `dissectable' closure algebra, showing that any such algebra contains copies of every finite closure algebra, and proving that any metric space without isolated points has a dissectable algebra of subsets. This work has been described \cite{john:elem01} as the first attempt to do pointless topology, a subject that has been a significant theme in the work of Bernard Banaschewski.

Our aim here is to generalise this theory to  a study of \emph{tangled closure}. In a topological space this operation assigns to each finite collection $\G$ of subsets a set $\C^t\G$ which the largest subset in which each member of $\G$ is dense. When $\G$ has one member, $\C^t\{\g\}$ is just the usual topological closure $\C\g$ of $\g$. In an order topology, determined by some quasi-ordering relation $R$, a point $x$ belongs to the tangled closure $\C^t\G$ iff there exists an `endless $R$-path' 
 $xRx_1\cdots x_nRx_{n+1}\cdots$ starting from $x$ such that the path enters each set belonging to $\G$ infinitely often.
 
 This order-theoretic interpretation has been used to model a propositional connective known as the \emph{tangle modality}, which was introduced by Dawar and Otto \cite{dawa:moda09}  in an analysis of logical formulas whose satisfaction is invariant under certain `bisimulation' relations between models.
 A well-known result of van Benthem  \cite{bent:moda76corr,bent:moda83} states that a first-order formula is invariant under bisimulations between arbitrary models iff that formula is equivalent to a formula of the basic language of propositional modal logic. This result continues to hold for bisimulation-invariance over any elementary class of models, such as the quasi-orderings, as well as over the class of all finite models. But on restriction to the class of all finite quasi-orderings (and some of its subclasses), the picture changes. 
 Propositional formulas involving the tangle modality, which are bisimulation-invariant, become first-order definable in this setting, and van Benthem's result no longer holds. 
 Instead, a first-order formula is bisimulation-invariant over the finite quasi-orderings iff it is equivalent to a formula of the language that enriches basic modal logic by the addition of the tangle modality. Moreover, \cite{dawa:moda09} showed   that the bisimulation-invariant fragment of monadic second order logic, which is equivalent over arbitrary models to the much more powerful modal mu-calculus, collapses over finite quasi-orderings to the first-order fragment,  so is also equivalent to the language with the tangle modality. The name `tangle' was introduced by Fern\'{a}ndez-Duque \cite{fern:tang11,fern:tang12} who axiomatised the tangle modal  logic of finite quasi-orderings. Subsequently we have made an extensive study \cite{gold:spat14,gold:fini16,gold:spat16,gold:tang16} of a range of logics with this connective.
 
 That accounts for the motivating origin of $\C^t\G$, but here we subject it to an abstract algebraic analysis, defining a tangled closure algebra as a pair  $(A,\C^t)$ with $\C^t$ an operation on finite subsets of a Boolean algebra $A$,  with the restriction of $\C^t$ to one-element sets being a closure operator $\C$. We require  $\C^t$  to satisfy equational conditions ensuring that $\C^t\G$ is the greatest fixed point of the function  $a\mapsto\meet_{\g\in\G}\C(\g\land a)$. We study homomorphisms and subalgebras of tangled closure algebras, and use the logical Lindenbaum-Tarski  algebra construction to produce freely generated tangled closure algebras.  Our main results extend those of McKinsey and Tarski by showing that if a tangled closure algebra  $(A,\C^t)$ is dissectable, then any finite tangled closure algebra can be isomorphically embedded into the \emph{relativised} algebra of all elements below some open  element of $(A,\C^t)$. Furthermore, if $(A,\C^t)$ is totally disconnected (e.g.\ the algebra of subsets of any zero-dimensional metric space without isolated points), then the embedding can be mapped into the relativisation to \emph{any} non-zero open element.
 
As is well known, every Boolean algebra $A$ has order-complete extensions, including the extension given by the Stone representation theory, and the MacNeille completion, which is a complete Boolean algebra $B$ extending $A$ with each element of $B$ being the join of a subset of $A$. A closure algebra  also has complete extensions of both these kinds. But in our final section we construct a tangled closure algebra that has no embedding into any complete tangled closure algebra at all. In particular, it cannot be represented as an algebra of subsets of a topological space.

\section{Tangled Closure}
Let $A$ be a Boolean algebra with signature $\land,\ \lor,\ -,0,1$. Define the Boolean implication operation in $A$ by
 $a\imp b=-a\lor b$, and put $a\liff b=(a\imp b)\land(b\imp a)$.
 Let $\join E$ and $\meet E$ denote the join and meet of a subset $E$ of $A$ when these exist. We sometimes write them as 
 $\join_A E$ and $\meet_A E$ to clarify which algebra they are being defined in.
 
 A \emph{closure operator} on $A$ is a function $\C:A\to A$ that is additive, normal,  inflationary and idempotent, i.e.\ satisfies the equational conditions
 $$
 \C(a\lor b)=\C a\lor\C b, 
 \quad \C 0=0, 
 \quad a\leq \C a
 = \C\C a.
 $$
$\C$ is then \emph{monotonic}, i.e.\ $a\leq b$ implies $\C a\leq\C b$, and \emph{finitely additive} in the sense that 
$\C\join\G=\join\{\C\g:\g\in\G\}$ for all finite $\G\sub A$. The pair $(A,\C)$ is called a \emph{closure algebra}. An element $a\in A$ is called \emph{closed} if $a=\C a$, which is equivalent to having $a=\C b$ for some $b$.

In a closure algebra, $\C$ has a dual \emph{interior} operation $\I:A\to A$ defined by  $\I a=-\C-a$. This is also mononotonic;  \emph{multiplicative} in the sense that  $\I\meet\G=\meet\{\I\g:\g\in\G\}$ for all finite $\G$; and has $\I 1=1$ and $\I\I a=\I a\leq a$. An element $a$ is called \emph{open} if $a=\I a$, which is equivalent to having $a=\I b$ for some $b$.

A basic property of all closure algebras that we make use of is that
\begin{equation}  \label{ICprop}
\I a\land\C b\leq\C(a\land b).
\end{equation}
In addition to the original paper \cite{mcki:alge44}, there is extensive information about closure algebras in Chapter III of \cite{rasi:math63}, where they are called \emph{topological Boolean algebras}. 

Let $\c P_{fin}A$ be the set of finite non-empty subsets of $A$. A function  $\C^t:\c P_{fin}A\to A$ induces a unary function
 $\C:A\to A$ by putting $\C a=\C^t\{a\}$, and hence a dual operation $\I$ that has  $\I a=-\C^t\{-a\}$.
We will write these operations as $\C^t_A$, $\C_A$, $\I_A$, when needing to distinguish which algebra we are in. 

 We say that 
$\C^t$ is a \emph{tangled closure operator}, and $(A,\C^t)$ is a \emph{tangled closure algebra},
if its induced $\C$ is a {closure operator }on $A$, and the following hold for all $\G\in\c P_{fin}A$ and $a\in A$:
\begin{description}
\item[Fix:]
$\C^t\G\leq  \meet_{\g\in\G}\C(\g\land\C^t\G)$, 
\item[Ind:]
$\I(a\imp \meet_{\g\in\G}\C(\g\land a))\land a\leq \C^t\G$.
\end{description}
These conditions are evidently equational, e.g.\ Fix is equivalent to
$
\C^t\G\land  \meet_{\g\in\G}\C(\g\land\C^t\G)=\C^t\G.
$
The pair $(A,\C)$ will be called the \emph{closure algebra reduct} of $(A,\C^t)$.

\begin{lemma} \label{lem:defCt}
In any tangled closure algebra, it holds in general that
\begin{equation}  \label{defCt}
\C^t\G  =\bigvee\{a\in A: a\leq\meet_{\g\in\G}\C(\g\land a)\}.
\end{equation}
\end{lemma}
\begin{proof}
Let $f_\G(a)= \meet_{\g\in\G}\C(\g\land a)$. This defines a function $f_\G:A\to A$ that is monotonic. Say that $a$ is a \emph{post-fixed point for} $\G$ if $a\leq f_\G(a)$.
Let 
$S_\G=\{a\in A: a\leq f_\G(a)\}$ be the set of all post-fixed points for $\G$. \eqref{defCt} asserts that $\C^t\G$ is the join of 
$S_\G$.

Now Fix states that $\C^t\G\leq f_\G(\C^t\G)$, hence $\C^t\G\in S_\G$.
But Ind implies that $\C^t\G$ is an upper bound of $S_\G$, for if $a\in S_\G$, then $\I(a\imp f_\G(a))=\I 1=1$, so Ind reduces in this case to the assertion that  $a\leq\C^t\G$.

Thus $\C^t\G$ is both a member of $S_\G$ and an upper bound of it, hence is its least upper bound.
\end{proof}

\begin{corollary}
$\C^t\G= \meet_{\g\in\G}\C(\g\land\C^t\G)$. Moreover $\C^t\G$ is the greatest (post-)fixed point of $f_\G$.
\end{corollary}

\begin{proof}
As above $\C^t\Gamma \leq  f_\Gamma(\C^t\Gamma)$, so monotonicity of $f_\G$ yields $f_\G\C^t\G\leq f_\G (f_\G\C^t)$, showing $f_\G\C^t\G$ is also a post-fixed point of $f_\G$, hence $ f_\G\C^t\G\leq \C^t\G$. Altogether then $\C^t\G=f_\G\C^t\G$, so $\C^t\G$ is a fixed point of $f_\G$. Since all such fixed points belong to $S_\G$, $\C^t\G$ is the greatest of them, as well as of the post-fixed points.
\end{proof}

Lemma \ref{lem:defCt} implies that $\C^t$ is uniquely determined by the unary $\C$ it induces. Furthermore:

\begin{theorem} \label{thm:compcla}
Any \textbf{complete} closure algebra $(A,\C)$ expands uniquely to a tangled closure algebra $(A,\C^t)$ inducing $\C$, by taking \eqref{defCt} as the \textbf{definition} of $\C^t\G$.  
\end{theorem}
\begin{proof}
For each $\G\in\c P_{fin}A$, define $\C^t\G=\join S_\G$, where $S_\G=\{a: a\leq f_\G(a)\}$  as above. Then we need to derive Fix and Ind for $\C^t$ thus defined.
First, if $a\in S_\G$, then $a\leq \C^t\G$, so $f_\G(a)\leq f_\G(\C^t\G)$. But $a\leq f_\G(a)$, so this shows that 
$a\leq f_\G(\C^t\G)$, for all $a\in S_\G$. Hence $\join S_\G\leq  f_\G(\C^t\G)$, i.e.\ $\C^t\G\leq f _\G(\C^t\G)$, which is Fix.

The derivation of Ind is more lengthy, and uses some basic properties of closure algebras. Given $\G$ and $a$, let 
$b=(a\imp \meet_{\g\in\G}\C(\g\land a))$. Ind asserts $\I b\land a\leq\C^t\G$, so to  prove this it is enough to show that $\I b\land a$ belongs to $S_\G$, i.e.\ $\I b\land a$ is a post-fixed point of $f_\G$. Taking an arbitrary $\g'\in\G$ we have

\begin{tabular}{ll}  
$\I b\land a$ &
\\
$= \I b\land (a\imp \meet_{\g\in\G}\C(\g\land a))\land a$
& as $\I b=\I b\land b$
\\
$\leq \I b\land \meet_{\g\in\G}\C(\g\land a)$
& by Boolean algebra
\\
$\leq \I\I b\land \C(\g'\land a)$
& $\I b=\I\I b$ and Boolean algebra
\\
$\leq \C(\I b\land \g'\land a)$
&  by \eqref{ICprop}
\end{tabular}

\noindent
This shows that $\I b\land a\leq \C( \g'\land \I b\land a)$ for all $\g'\in\G$, hence
$$
\I b\land a\leq \meet_{\g\in\G}\C( \g\land \I b\land a).
$$
But that says  $\I b\land a\leq f_\G(\I b\land a)$, i.e. $\I b\land a\in S_\G$, hence 
$\I b\land a\leq\join S_\G=\C^t\G$, which is Ind.

It remains to show that $\C$ is the closure operator induced by $\C^t$. Let $\G$ be any singleton $\{\g\}$. Then if $a\in S_{\{\g\}}$, $a\leq \C(\g\land a)\leq \C\g$. So $\C\g$ is an upper bound of $S_{\{\g\}}$. But 
$\C\g\leq\C(\g\land\C\g)$, since $\g\leq\C\g$, showing that $\C\g$ also belongs to $S_{\{\g\}}$.
Hence $\C\g=\join S_{\{\g\}}=\C^t\{\g\}$ as required.
\end{proof}

\begin{example}[Spatial Tangled Closure] \em  \label{ex:spatial}
The paradigm of a closure algebra is $(A_S,\C_S)$ where $S$ is any topological space. Here $A_S$ is the Boolean powerset algebra of all subsets of $S$, and $\C_S(a)$ is the topological closure of the set $a\sub S$, the intersection of all closed supersets of $a$. This is a complete closure algebra in which $\join E=\bigcup E$ and $\meet E=\bigcap E$ for all $E\sub A_S$. By Theorem \ref{thm:compcla}, $\C_S$ has a unique expansion to a tangled closure operator 
$\C^t_S$.  A point belongs to $\C_S^t\G$ iff it belongs to some set $a$ such that for all $\g\in\G$, $a\sub \C_S(\g\cap a)$, so $\g$ is dense in $a$ in the sense that any open neighbourhood of any point of $a$ contains a point in $\g$ and $a$. Since $\C^t\G$ is the greatest post-fixed point for $\G$, it is the largest set in which every member of $\G$ is dense.
\end{example}

\begin{example}[Quasi-orders and Alexandroff Spaces]\em    \label{ex:quoset}
A quasi-order is  a reflexive transitive binary relation $R$ on a set $S$. The pair  $(S,R)$ is a  \emph{quasi-ordered set}. Each $x\in S$ has the set $R(x)=\{y:xRy\}$ of $R$-\emph{successors}. Then $y\in R(x)\cap R(z)$ implies $y\in R(y)\sub R(x)\cap R(z)$, so the collection $\{R(x):x\in S\}$ of successor sets is a basis for a topology on $S$, the \emph{Alexandroff topology}. Its open sets are the \emph{up-sets}, those subsets $a$ of $S$ such that are closed upwards in the quasi-ordering in the sense that $x\in a$ implies $R(x)\sub a$. Its closed sets are the \emph{down-sets}, the sets $a$ for which $xRy\in a$ implies $x\in a$.
Its closure operator $\C_R$ has $\C_R(a)=R^{-1}(a)=\{x:\exists y(xRy\in a)\}$, giving the closure algebra $(A_S,\C_R)$.
Hence by the preceding Example, the tangled closure operator $\C^t_R$ of this space has
$$
\C^t_R\G=\bigcup \{a\sub S: a\sub{{\bigcap}_{\g\in\G}R^{-1}(\g\cap a)}\}.
$$
To give an alternative characterisation of $\C^t_R$, define an \emph{endless $R$-path} to be a sequence 
$\{x_n:n<\omega\}$ in $S$ such that $x_nRx_{n+1}$ for all $n$. (The terms $x_n$ of the sequence need not be distinct. Indeed $S$ may be finite.) Then it can be shown that
\begin{quote}
$x\in\C^t_R\G$ iff there exists an endless $R$-path  $\{x_n:n<\omega\}$ in $S$ with $x=x_0$, such that for each $\g\in \G$ there are infinitely many $n<\omega$ such that $x_n\in \g$.
\end{quote}
(see  \cite[\S4.1]{fern:tang12}). We use this characterisation in several places below.
  \qed
\end{example}

The next theorem records properties that will be used in Section \ref{sec:nocomp} in constructing a tangled closure algebra with no complete extension.

\begin{theorem}  \label{thm:cong}
In any tangled closure algebra $ (A,\C^t)$, the following hold for all $\G\in\c P_{fin}A$.
\begin{enumerate}[\rm(1)]
\item
$\C^t\G$ is closed, i.e.\ $\C\C^t\G=\C^t\G$.
\item 
If $\G'=\{\g':\g\in\G\}\sub A$, then
$$
\meet_{\g\in\G}\I(\g\liff\g')\leq\I\big(\C^t\G\liff \C^t\G'\big).
$$
\end{enumerate}
\end{theorem}

\begin{proof}
\begin{enumerate}[\rm(1)]
\item 
We have $\C^t\G\leq\C\C^t\G$ as $\C$ is a closure operator, so we need to show the reverse inequality $\C\C^t\G\leq\C^t\G$. For this it suffices by Lemma \ref{lem:defCt} to show that $\C\C^t\G$ is a post-fixed point for $\G$.
Now for any $\g\in\G$, by Fix and $\C$-monotonicity 
$\C\C^t\G\leq  \C\C(\g\land\C^t\G)=\C(\g\land\C^t\G)$.
But by closure algebra properties $\C(\g\land\C^t\G)\leq\C(\g\land\C\C^t\G)$.
Altogether this implies that
$\C\C^t\G\leq  \C(\g\land\C\C^t\G)$ for all $\g\in\G$. Hence $\C\C^t\G\in S_\G$ as required.

\item
Let $a=\meet_{\g\in\G}\I(\g\liff\g')$. As $\I$ is multiplicative, $a=\I\meet_{\g\in\G}(\g\liff\g')$, so $a$ is open and therefore $a=\I a$. Now for each $\g\in\G$, using Fix we have
$$
a\land\C^t\G \leq \I(\g\liff\g')\land \C(\g\land\C^t\G) \leq \C((\g\liff\g')\land \g\land\C^t\G)
$$
by \eqref{ICprop}. Since $(\g\liff\g')\land \g\leq\g'$ and $\C$ is monotonic, this implies $a\land\C^t\G\leq  \C(\g'\land\C^t\G)$.
Thus $a\land\C^t\G\leq  \meet_{\g\in\G}\C(\g'\land\C^t\G)$, so $a\leq\C^t\G\imp  \meet_{\g\in\G}\C(\g'\land\C^t\G)$. Hence
$$
a=\I a\leq\I(\C^t\G\imp  \meet_{\g\in\G}\C(\g'\land\C^t\G)).
$$
Then $a\land\C^t\G\leq \I(\C^t\G\imp  \meet_{\g\in\G}\C(\g'\land\C^t\G))\land\C^t\G\leq \C^t\G'$ by Ind for $\G'$. It follows that 
$a\leq \C^t\G\imp\C^t\G'.$
Interchanging $\G$ and $\G'$ here, and using $\g\liff\g'=\g'\liff\g$, we likewise show $a\leq \C^t\G'\imp\C^t\G$. Hence
$a\leq \C^t\G\liff\C^t\G'$. Therefore
$
a=\I a\leq \I(\C^t\G\liff\C^t\G'),
$
which is the desired result.
\end{enumerate}
\end{proof}

\section{Homomorphisms, Subalgebras, Free Algebras}  \label{hsf}

A \emph{homomorphism} $f:(A,\C_A^t)\to(B,\C_B^t)$ between algebras of the type of tangled closure algebras  is a  Boolean algebra homomorphism $f:A\to B$ that preserves the $\C^t$-operations in the sense that
$$
f(\C^t_A\G)=\C_B^t\{f\g:\g\in\G\}.
$$
If $f$ is injective we call it an \emph{embedding}. If it is surjective, then it preserves validity of equations, hence if                       $(A,\C_A^t)$ is a tangled closure algebra, then so is $(B,\C_B^t)$. If $f$ is bijective then it is an \emph{isomorphism}.

A homomorphism of tangled closure algebras preserves the associated closure operators, meaning that   $f(\C_A(a))=\C_B f(a)$. In general a Boolean homomorphism $f:A\to B$ that is a closure algebra homomorphism in this sense need not preserve tangled closure, as we will see later in Section \ref{sec:nocomp}. However, if $f$  is a closure algebra  \emph{isomorphism}  from $(A,\C_A)$ onto $(B,\C_B)$, then it will preserve tangled closure  and be a tangled closure algebra isomorphism from 
$(A,\C_A^t)$ onto $(B,\C_B^t)$. This follows by \eqref{defCt}, since Boolean isomorphisms preserve all existing joins.

\begin{theorem}  \label{thm:repfinite}
Any finite tangled closure algebra $(A,\C_A^t)$ is isomorphic to the powerset algebra $(A_S,\C^t_R)$ of some finite quasi-ordered set $(S,R)$ (see Example \ref{ex:quoset}).
\end{theorem}

\begin{proof} 
Being finite, $A$ is isomorphic to the powerset algebra $A_S$ where $S$ is the set of atoms of $A$. The closure operator $\C_A$ induced by $\C^t_A$ is transferred by the isomorphism to a closure operator $\C'$ on $A_S$. Here $\C'$ is equal to the operator $\C_R=R^{-1}$ of a quasi-order on $S$ defined by $xRy$ iff $x\in\C'\{y\}$.
This follows from work in \cite[Section 3]{jons:bool51} on complete and atomic algebras, and is set out explicitly in \cite[Lemma 1]{dumm:moda59}. 

Since the closure algebras $(A,\C_A)$ and $(A_S,\C_R)$ are isomorphic, it then follows that $(A,\C_A^t)$ and  $(A_S,\C^t_R)$ are isomorphic, as noted above.
\end{proof}

Another case in which a closure algebra homomorphism between tangled closure algebras must preserve tangled closure occurs when the domain of the homomorphism is \emph{finite}, as we now show.

\begin{theorem}  \label{findom}
Let $(A,\C_A^t)$ and $(B,\C_B^t)$ be tangled closure algebras and $f:A\to B$ be a closure algebra homomorphism between the associated closure algebra reducts $(A,\C_A)$ and $(B,\C_B)$. Suppose $A$ is finite. Then $f$ preserves the tangled closure operations $\C_A^t$ and $\C_B^t$.
\end{theorem}

\begin{proof}
We need to show that if $\G\in\c P_{fin}A$, then $f\C^t_A\G=\C^t_Bf\G$, where $f\G=\{f\g:\g\in\G\}$. We use the fact that $f$ is monotonic and preserves finite meets and closure operators. Applying this to Fix for $\C^t_A\G$ gives that in $B$,
$$
f\C^t_A\G\leq  \meet_{\g\in\G}\C_B(f\g\land f\C^t_A\G).
$$
This means that $f\C^t_A\G$ is a post-fixed point for $f\G$ in $B$, so   by Lemma \ref{lem:defCt}, $f\C^t_A\G\leq\C^t_Bf\G$.

For the reverse inequality $\C^t_Bf\G\leq f\C^t_A\G$, let $D=\{a\in A:\C^t_Bf\G\leq fa\}$. Put $d=\meet D$, which exists in $A$ as $A$ is finite. Then  in $B$ we have
\begin{equation}  \label{ineqfd}
\C^t_Bf\G\leq \meet\{fa:a\in D\} =fd
\end{equation}
as $f$ preserves finite meets.
Now by  Fix for $\C^t_Bf\G$, \eqref{ineqfd} and preservation properties of $f$ we get
$$
\C^t_Bf\G \leq  \meet_{\g\in\G}\C_B(f\g\land \C^t_Bf\G) \leq 
 \meet_{\g\in\G}\C_B(f\g\land fd) =
f\big( \meet_{\g\in\G}\C_A(\g\land d)\big).
$$ 
This shows that $\meet_{\g\in\G}\C_A(\g\land d)\in D$. Hence $d=\meet D\leq \meet_{\g\in\G}\C_A(\g\land d)$, proving that $d$ is a post-fixed point for $\G$. Thus $d\leq \C^t_A\G$ by  Lemma \ref{lem:defCt}. Then by \eqref{ineqfd} and monotonicity of $f$,
$$
\C^t_Bf\G\leq fd\leq f\C^t_A\G,
$$
completing the proof that $\C^t_Bf\G\leq f\C^t_A\G$ and hence $\C^t_Bf\G= f\C^t_A\G$.\footnote{This proof can be adapted to yield the following  result. If $\a:A\to A$ and $\b:B\to B$ are monotonic functions on complete lattices $A$ and $B$, and $f:A\to B$ is  a complete lattice homomorphism
such that $f\circ\alpha=\beta\circ f$, then $f$ preserves the greatest and least fixed points of $\a$ and $\b$.}
\end{proof}

We will say that
$(A,\C_A^t)$ is a \emph{subalgebra} of $(B,\C_B^t)$ if $A$ is a Boolean subalgebra of $B$ that is closed under $\C^t_B$, i.e.\ $\C^t_B\G\in A$ for all $\G\in\c P_{fin}A$, and $\C_A^t$ is the restriction of $\C_B^t$ to $A$. Equivalently this means that $A\sub B$ and the inclusion $A\hookrightarrow B$ is  a homomorphism $(A,\C_A^t)\to(B,\C_B^t)$ as above. 
This implies that the reduct $(A,\C_A)$ is a subalgebra of $(B,\C_B)$. But we will  see in Section \ref{sec:nocomp} that it is possible to have 
$(A,\C_A)$  a subalgebra of $(B,\C_B)$ while $(A,\C_A^t)$ is not a subalgebra of $(B,\C_B^t)$.
 
 We also need the notion of the   \emph{relativisation} of an algebra to one of its elements. This abstracts from  the notion of a topological subspace, i.e.\ the relativisation of a topology to a subset.
To describe it, let $(A,\C_A^t)$ be an abstract  tangled closure algebra with closure algebra reduct $(A,\C_A)$. If $\a\in A$, let   $A_\a=\{b\in A:b\leq\a\}$ be the Boolean algebra of elements below $\a$, in which joins and meets are the same as in $A$, and the complement of $b$ in $A_\a$ is $\a-b=\a\land -b$. The implication operation $\imp_\a$ of $A_\a$ has 
$$
b\imp_\a c=(\a-b)\lor c\leq -b\lor c = b\imp c.
$$ 
A closure operator $\C_{\a}$ is 
defined on $A_\a$ by putting  $\C_{\a}b: =\a\land\C_A b$. The dual operator $\I_\a$ to $\C_\a$ has the property that if $\a$ is an open element of $A$, i.e.\ $\I_A\a=\a$, then $\I_\a b=\I_A b$ for all $b\in A_\a$ \cite[p.~96]{rasi:math63}.
Define an operation $\C_\a^t$ on $\c P_{fin} A_\a$ by putting  $\C_\a^t\G:=\a\land\C_A^t\G$.  $(A_\a,\C_\a^t)$ is the \emph{relativisation of  $(A,\C_A^t)$  to $\a$}.

\begin{theorem}  \label{openreduct}
 If $\a$ is open, then $(A_\a,\C_\a^t)$ is a tangled closure algebra with closure algebra reduct $(A_\a,\C_{\a})$.
\end{theorem}

\begin{proof}
 $\C_\a^t$ induces the unary operation $ b\mapsto\a\land\C_A b$, which is the closure operator $\C_\a$ above. $\C_\a^t$ satisfies Fix, since for all finite $\G\sub A_\a$ and all $\g\in\G$,  using Fix for $\C_A^t$ shows that 
$$
\C_\a^t\G=\a\land\C_A^t\G\leq \a\land \C_A(\g\land\C_A^t\G)
= \a\land \C_A(\g\land\a\land \C_A^t\G)
= \a\land \C_A(\g\land\C_\a^t\G)
=\C_\a(\g\land\C_\a^t\G).
$$
To show $\C_\a^t$ satisfies Ind, we need the assumption that $\a$ is open,  implying that $\I_\a$ is the restriction of $\I_A$ to $A_\a$. Let $x=\I_\a(b\imp_\a \meet_{\g\in\G}\C_\a(\g\land b))\land b$ where $b\leq \a$. Then 
\begin{align*}
x&= \I_A(b\imp_\a \meet_{\g\in\G}\C_\a(\g\land b))\land b
\\
&\leq \I_A(b\imp \meet_{\g\in\G}\C_A(\g\land b))\land b
\\
&\leq \C_A^t\G
\end{align*}
by Ind for $\C_A^t$. Hence $x\leq\a\land\C_A^t\G = \C_\a^t\G$, which gives Ind for $\C_\a^t$.
\end{proof}

A \emph{free} tangled closure algebra over any set $V$ can be constructed by using a propositional modal logic and the standard Lindenbaum-Tarski algebra construction. To outline this, take an arbitrary $V$ and regard its members as (propositional) variables that can range over the elements of an algebra. From these variables we construct \emph{formulas} $\ph,\psi,\dots$ using
\begin{itemize}
\item
the Boolean connectives  $\land,\ \lor,\ \neg,\ \to,\ \leftrightarrow$, and a constant $\bot$, interpreted as the corresponding  operations in a Boolean algebra;
\item 
unary modalities $\di$ and $\bo$ interpreted as $\C$ and $\I$;
\item
a new connective $\dit$, interpreted as $\C^t$, which provides  formation of a formula  $\dit\G$ for each finite non-empty set $\G$ of formulas. 
\end{itemize}
We denote by S4$t$ be the propositional logic obtained by adding to a suitable axiomatisation of the (non-modal) two-valued propositional calculus the axiom schemes
\begin{description}
\item[K:]
$\Box(\ph\to\psi)\to(\Box\ph\to\Box\psi)$
\item[T:]
$\ph\to\di\ph$
\item[4:]
 $\di\di\ph\to\di\ph$
\item[Fix:]
$\dit\G\to  \di(\g\land\dit\G)$, \quad all $\g\in\G$,
\item[Ind:]
$\Box(\ph\to \bigwedge_{\g\in\G}\di(\g\land\ph))\to(\ph\to\dit\G)$,
\end{description}
and the inference rule of  $\bo$-generalisation (from $\ph$ infer $\bo\ph$). We write S4$t\vdash\ph$ to mean that formula $\ph$ is derivable as a theorem of this logic, which is studied in detail in \cite{gold:spat14,gold:fini16,gold:spat16}.

An equivalence relation $\equiv$ on formulas is defined by putting $\ph\equiv\psi$ iff S4$t\vdash\ph\leftrightarrow\psi$. If 
$\ab{\ph}=\{\psi:\ph\equiv\psi\}$ is the equivalence class of $\ph$, then the Lindenbaum-Tarski algebra of S4$t$ is the set
$A_{t}=\{\ab{\ph}: \ph \text{ is a formula}\} $ of all equivalence classes, with the operations
\begin{align*}
\ab{\ph}\land \ab{\psi} &=\ab{\ph\land\psi} \\
\ab{\ph}\lor \ab{\psi} &=\ab{\ph\lor\psi}\\
-\ab{\ph} &=\ab{\neg\ph}\\
0 &=\ab{\bot}\\
1&=\ab{\neg\bot}\\
\C^t_{A_t}\{\ab{\ph}:\ph\in\G\}  &=\ab{\dit\G}.
\end{align*}
$(A_t,\C^t_{A_t})$ is a well-defined tangled closure algebra having an injective function $\eta:V\to A_t$ given by $\eta(v)=\ab{v}$. This is for the most part standard theory \cite[\S\S VI.10, XI.7]{rasi:math63}. That $\C^t_{A_t}$ is well-defined follows because if  $\G'=\{\ph':\ph\in\G\}$ and S4$t\vdash\ph\leftrightarrow\ph'$ for all $\ph\in\G$, then
S4$t\vdash\dit\G\leftrightarrow\dit\G'$. The axioms \textbf{Fix} and \textbf{Ind} for S4$t$ ensure that $\C^t_{A_t}$ is a tangled closure operator.

The image $\{\ab{v}:v\in V\}$ of $\eta$ generates the algebra $(A_t,\C^t_{A_t})$, which is free over $V$ in the sense that for any tangled closure algebra $(A,\C_A^t)$ and any function $f:V\to A$, there is a unique tangled closure algebra homomorphism 
$f':(A_t,\C^t_{A_t})\to(A,\C^t_{A})$ such that $f'\circ\eta=f$. The function $f$ itself is extended to map all formulas into $A$ by interpreting the connectives by the corresponding operations of $(A,\C^t_A)$, and then $f'$ is defined by putting 
$f'\ab{\ph}=f(\ph)$. Identifying $v$ with $\eta(v)$ allows us to  view $V$ as a subset of $A_t$ that freely generates
$(A_t,\C^t_{A_t})$.

\begin{remark}\em  \label{infsig}
A tangled closure algebra differs from the type of algebra conventionally studied in universal algebra, since the operation $\C^t$ is not finitary, i.e.\ not $n$-ary for any $n<\omega$. But it  gives rise to the sequence of finitary operations $\{\C^t_n:n\geq 1\}$, where  $\C^t_n$ is the $n$-ary operation defined by $\C^t_n(\vec{a}{n})=\C^t\{\vec{a}{n}\}$. We could  define a tangled closure algebra as a conventional algebra with infinite signature, having the form $(A,\{\C^t_n:n\geq 1\})$, satisfying axioms Fix$_n$ and Ind$_n$ stated in terms of $\C^t_n$ for each $n\geq 1$, and satisfying axioms 
$
\C^t_n(\vec{a}{n})=\C^t_n(a_{\sigma 1},\dots,a_{\sigma n})
$
expressing the invariance of $\C^t_n$ under any permutation of its arguments. It is evident that this alternative approach is equivalent to the presentation we have given here. But it helps clarify that the class of tangled closure algebras is an equational class, or variety, in the traditional sense. 
\qed
\end{remark}
The logic S4$t$ has the finite model property: any non-theorem of the logic is falsifiable in the powerset algebra 
$(A_S,\C^t_R)$ of some finite quasi-ordered set $(S,R)$ (see \cite{fern:tang11,gold:spat14,gold:fini16} for a proof). From this it can be concluded that the variety of tangled closure algebras is generated by its finite members.

\section{Dissectable Algebras}  \label{sec:dissect}

A closure algebra $(B,\C_B)$ is \emph{dissectable} if for any non-zero open element $\a$ of $B$, and any natural numbers  $r$ and $s$,  there exist non-zero elements $\vec{\a}{r},\Vec{\b}{s}$ of $B$ such that
\begin{itemize}
\item 
these elements form a partition of $\a$, i.e.\ they are pairwise disjoint (any two have meet 0) and the join of all of them is $\a$;\footnote{When $r=0$, the sequence $\vec{\a}{r}$ is empty.}
\item
$\vec{\a}{r}$ are all open; \en and
\item
for all $i\leq r$ and $j\leq s$, \en 
$
\C_B\a_i-\a_i=\C_B\beta_j=\C_B\a-(\a_1\lor\cdots\lor\a_r).
$
\end{itemize}
Originally Tarski formulated the dissectability property with $s=0$, and proved that this holds for the powerset algebra of the real line and of its  dense-in-themselves subspaces. Density-in-itself means that there are no isolated points, i.e.\ no open singletons. Samuel Eilenberg then proved that the property holds for any separable dense-in-itself metric space, and this was presented in \cite{tars:auss38}. The more general formulation with arbitrary finite $s$ was given in \cite{mcki:alge44}, where it was shown to hold for  separable dense-in-themselves metric spaces. Another proof was given in \cite{rasi:math63} that eliminated the separability restriction. New kinds of dissectability theorems along these lines are presented in \cite{gold:spat14,gold:spat16,gold:tang16}.

It was shown in \cite{mcki:alge44} that if $(B,\C_B)$ is dissectable then every finite closure algebra is isomorphic to a subalgebra of the relativised algebra $(B_\a,\C_\a)$ for some non-zero open element $\a$ of $B$, and that any such $(B_\a,\C_\a)$ is itself dissectable. Moreover,  a \emph{well-connected} finite closure algebra is embeddable into $(B_\a,\C_\a)$ for \emph{every} non-zero open $\a$.
Well-connectedness means that $\C a\land\C b=0$ implies $a=0$ or $b=0$. Equivalently, it means that the meet of any two non-zero closed elements is non-zero. In a finite closure algebra, this means that there is a \emph{least} non-zero closed element, a property called \emph{strong compactness} in \cite[p.~110]{rasi:math63}. For the powerset closure algebra $(A_S,\C_R)$ of a quasi-order set $(S,R)$, as in Example \ref{ex:quoset}, this means that the quasi-order is \emph{point-generated} in the sense that there is a point $x\in S$ such that $R(x)=S$, so that every $y\in S$ has $xRy$. To see why, let $a$ be a least non-empty closed subset of $S$ in the Alexandroff topology. Take any $x\in a$. Then for any $y\in S$, the set $\{z:zRy\}$ is closed and non-empty, so includes $a$, showing that $xRy$. Hence $R(x)=S$. Conversely, if $R(x)=S$, then the set $\{z:zRx\}$ is a non-empty closed set included in all others. In summary: if $S$ is finite, then $(A_S,\C_R)$ is well-connected iff $(S,R)$
 is point-generated.
 
 Using our result from the previous section on homomorphisms with finite domains, we can readily lift the McKinsey-Tarski analysis  to tangled closure algebras.

\begin{theorem}  \label{thm:pointembed}
Let $(B,\C_B^t)$ be a tangled closure algebra whose closure algebra reduct $(B,\C_B)$ is dissectable. Then any finite tangled closure algebra with a well-connected closure algebra reduct is isomorphically embeddable into the relativised algebra $(B_\a,\C_\a^t)$ of any non-zero open element $\a$ of $(B,\C_B)$.
\end{theorem}

\begin{proof}
Let be any non-zero open element  of $(B,\C_B)$.  By Theorem \ref{openreduct} $(B_\a,\C_\a^t)$ is a tangled closure algebra, with closure algebra reduct $(B_\a,\C_\a)$.

Now let $(A,\C^t_A)$ be a finite tangled closure algebra whose closure algebra reduct $(A,\C_A)$ is well-connected. Then
by  \cite[Theorem 3.7]{mcki:alge44} there is a closure algebra embedding $f:(A,\C_A)\to (B_\a,\C_\a)$. By our Theorem \ref{findom} this $f$ preserves the tangled closure operations $\C_A^t$ and $\C_\a^t$, so provides the result.
\end{proof} 

Note that by putting $\a=1$ in this Theorem, so that $B_\a=B$, we conclude that any well-connected finite tangled closure algebra is isomorphic to a subalgebra of $(B,\C^t_B)$ itself.
We can now apply this result to show that any finite tangled closure algebra has \emph{some} embedding into a relativised algebra of any dissectable tangled closure algebra.

\begin{theorem}
Let $(B,\C_B^t)$ be a dissectable tangled closure algebra. Then any finite tangled closure algebra  is isomorphically embeddable into the relativised algebra $(B_\a,\C_\a^t)$ of some open element $\a$ of $(B,\C_B)$.
\end{theorem}

\begin{proof}
By Theorem \ref{thm:repfinite} it suffices to give the proof for  finite algebras of the form $(A_S,\C^t_R)$. If $(S,R)$ is point-generated, the result follows from Theorem \ref{thm:pointembed}. Otherwise, we add a generating point. Let $x$ be any object not in $S$, put $S^*=S\cup\{x\}$, and let $R^*=R\cup(\{x\}\times S^*)$. Then $(S^*,R^*)$ is a quasi-ordered set point-generated by $x$, with no member of $S$ being $R^*$-related to $x$. The finite tangled closure algebra $(A_{S^*},\C^t_{R^*})$ is well-connected, so by Theorem \ref{thm:pointembed} with $\a=1$, there is a tangled closure embedding $h: (A_{S^*},\C^t_{R^*})\to (B,\C_B^t)$. The image 
$(B',\C^t_{B'})$ of $h$ is a tangled closure subalgebra of $(B,\C_B^t)$ isomorphic to $(A_{S^*},\C^t_{R^*})$, where 
$\C^t_{B'}$ is the restriction of $\C^t_{B}$ to $B'=h(B)$.

Now $S$ is a subset of $S^*$ that is closed upwards under $R^*$, so $S$ is an open element of $(A_{S^*},\C_{R^*})$, i.e.\ 
$\I_{R^*}(S)=S$. But $h$ preserves the interior operations $\I_{R^*}$ and $\I_B$, so then $h(S)$ is an open element of $(B, \C_B)$, i.e.\
$\I_{B}h(S)=h(S)$. Let $\a=h(S)\in B'$. Then as  $(A_{S^*},\C^t_{R^*})$ is isomorphic to $(B',\C^t_{B'})$ under $h$, the relativisation of  $(A_{S^*},\C^t_{R^*})$ to $S$ is isomorphic to the relativisation of $(B',\C^t_{B'})$ to $\a$, which is a subalgebra of the relativisation $(B_\a,\C_\a^t)$ of $(B,\C_B^t)$ to the open element $\a$.

But the relativisation of  $(A_{S^*},\C^t_{R^*})$ to $S$ is exactly  $(A_S,\C^t_R)$. For, the relativisation $(A_{S^*})_S$ of 
the powerset algebra $A_{S^*}$ of $S^*$ to $S$ is just the powerset algebra $A_{S}$ of $S$. Also, the relativisation of 
$\C^t_{R^*}$ to $S$ is the map $\G\mapsto  S\cap  \C^t_{R^*}\G$    for $\G\sub A_S$. But 
$S\cap  \C^t_{R^*}\G=\C^t_{R}\G$ because $S$ is closed upwards under $R^*$ and an endless $R^*$-path that starts in $S$ must remain in $S$ and be an endless $R$-path.

Altogether then, this shows that $(A_S,\C^t_R)$ is isomorphic to a subalgebra of $(B_\a,\C_\a^t)$.
\end{proof}

The proof of  Theorem \ref{thm:pointembed} can be extend to all finite tangled closure algebras if the dissectable algebra $(B,\C_B)$ is assumed to be \emph{totally disconnected}, which means that every non-zero open element is the join of two disjoint non-zero open elements. The totally disconnected dissectable algebras include the closure algebras of all dense-in-themselves metric spaces that are totally disconnected in the spatial sense that distinct points can be separated by a clopen set. Examples of such spaces include the rational line, the Cantor space and the Baire space $\omega^\omega$.

\begin{theorem}
Let $\a$ be any non-zero open element of  a tangled closure algebra $(B,\C_B^t)$ whose closure algebra reduct is totally disconnected and dissectable.
Then any finite tangled closure algebra  is isomorphically embeddable into the relativised algebra $(B_\a,\C_\a^t)$
\end{theorem}

\begin{proof}
The reduct $(B_\a,\C_\a)$ is totally disconnected and dissectable, so by \cite[Theorem 3.8]{mcki:alge44}, if $(A,\C^t_A)$
is any finite tangled closure algebra, there is  a closure algebra embedding $f:(A,\C_A)\to (B_\a,\C_\a)$. By Theorem \ref{findom} this $f$ preserves the tangled closure operations $\C_A^t$ and $\C_\a^t$, so provides the result.
\end{proof}
 
\section{No Completion}  \label{sec:nocomp}

A completion of a Boolean algebra $A$ is any \emph{complete} Boolean algebra $B$ extending $A$, i.e.\ having $A$ as a subalgebra, such that each member of $B$ is the join of a set of members of $A$. This last condition is equivalent to $A$ being \emph{dense} in $B$ in the sense that each non-zero member of $B$ is above some non-zero member of $A$. It implies that $B$ is a \emph{regular} extension of $A$, i.e.\ the inclusion $A\hookrightarrow B$ preserves any  joins (hence meets) that exist in $A$, so that if $a=\join_A E$ in $A$, then $a=\join_B E$ in $B$. Any Boolean algebra $A$ has a completion, and any two completions of $A$ are isomorphic by a function that is the identity on $A$ (see e.g.\ \cite{siko:bool64,giva:intr09,dave:intr90}). This unique-up-to-isomorphism algebra is often called the \emph{MacNeille completion} of $A$, after its construction in \cite{macn:part37}. It has various abstract characterisations, some due to Banaschewski \cite{bana:hull56,bana:cate67}.

If $(A,\C_A)$ is a closure algebra and $B$ is any complete extension of $A$, then $\C_A$ can be extended to a closure operator on $B$ by putting
\begin{equation}  \label{complC}
\C_B b=\meet_B\{\C_A a : b\leq a\in A\}
\end{equation}
for all $b\in B$. This definition   was given in \cite{mcki:alge44} where it was applied to the Stone representation of $A$ to   lift $\C_A$ to the powerset algebra of the representing set, ultimately showing that any closure algebra is embeddable into the complete algebra of subsets of some topological space. It was later used in \cite{rasi:alge51} to extend  $\C_A$ to the MacNeille completion of $A$, applying this to construct a regular complete extension of any Heyting algebra, and then using the regularity to obtain completeness theorems in algebraic semantics for versions of intuitionistic logic and the modal logic S4 with first-order quantifiers. In more recent literature on MacNeille completions \cite{theu:macn07}, $\C_B$ as given by \eqref{complC} is called the \emph{upper MacNeille extension} of $\C_A$.

There is no unique definition of MacNeille extension for operations on Boolean algebras. Monk \cite{monk:comp70} showed that for algebras, such as cylindric algebras, in which the operations are completely additive (preserve all joins), it is fruitful to use the \emph{lower MacNeille extension} which lifts an operation $\mathbf{O}_A$ to the operation
$\mathbf{O}_B b=\join_B\{\mathbf{O} a : b\geq a\in A\}$.

We now define a \emph{completion} of a closure algebra $(A,\C_A)$ to be a closure algebra $(B,\C_B)$ such that $B$ is a Boolean completion of $A$, $(A,\C_A)$ is a subalgebra of $(B,\C_B)$, and \eqref{complC} holds for each $b\in B$.

 \begin{theorem} \label{thm:closurecomp}
Any  closure algebra $(A,\C_A)$ has a completion, and any two such completions are isomorphic by a function that is the identity on $A$.
\end{theorem}
\begin{proof}
Let $B$ be a Boolean completion of $A$, and \emph{define} a closure operator $\C_B$ on $B$ by \eqref{complC}. Then $\C_B a=\C_A a$ for $a\in A$, so $(A,\C_A)$ is a subalgebra of $(B,\C_B)$ and $(B,\C_B)$ is a completion of $(A,\C_A)$. If $(B',\C_{B'})$ is another one, then there is a Boolean isomorphism $f:B\to B'$ that is the identity on $A$. Hence $f$ preserves joins and meets, and for any $b\in B$ and $a\in A$, we have $b\leq a$ iff $f(b)\leq a$. Then we can shown that $f$ preserves closure operators as follows.

\begin{tabular}{ll} 
$f(\C_B b)$ &
\\
$=f\meet_B\{\C_A a : b\leq a\in A\}$ &by \eqref{complC},
\\
$=\meet_{B'}\{f(\C_A a) : b\leq a\in A\}$    &as $f$ preserves meets,
\\
$=\meet_{B'}\{\C_A a : f(b)\leq a\in A\}$     &as $f$ fixes $A$
\\
$= \C_{B'} f(b)$  &by \eqref{complC} for $B'$.
\end{tabular}

\noindent
Thus $f$ is a closure algebra isomorphism.
\end{proof}

It would thus seem natural to define a completion of a tangled closure algebra $(A,\C_A^t)$ to be a tangled closure algebra $(B,\C_B^t)$ such that
\begin{enumerate}[(i)]
\item 
$(A,\C_A^t)$ is a subalgebra of $(B,\C_B^t)$,
\item
the closure algebra $(B,\C_B)$ induced by $\C^t_B$ is a completion of the
 closure algebra $(A,\C_A)$ induced by $\C^t_A$, 
\end{enumerate}
and perhaps some other conditions as well. However, we will now construct a tangled closure algebra $(A,\C_A^t)$ for which there is no \emph{complete} tangled closure algebra $(B,\C_B^t)$ satisfying (i), let alone (i) and (ii).

 \begin{lemma}  \label{existsA0}
There exists a tangled closure algebra $(A_0,\C^t)$ having a subset $\{p_n:n<\omega\}\cup\{q\}$ and an ultrafilter $x_0$ such that $\C^t\{q,-q\}\notin x_0$ while $\Sigma\sub x_0$, where
 $$
 \Sigma=\{p_0\}\cup\{\I(p_{2n}\Rightarrow \C(p_{2n+1}\land q)),\I(p_{2n+1}\Rightarrow\C(p_{2n+2}\land -q)):n<\omega\}.
 $$
 \end{lemma}
\begin{proof}
Let $(A_0,\C^t)$  be the free tangled closure algebra generated by a set $\{p_n:n<\omega\}\cup\{q\}$ of distinct elements.
This exists as explained in Section \ref{hsf}.
Let $a_{n}$ be $\I(p_{n}\Rightarrow \C(p_{n+1}\land q))$ if $n$ is even,  and 
 $\I(p_{n}\Rightarrow \C(p_{n+1}\land  -q))$ if $n$ is odd, where $\C$ is the closure algebra reduct of $\C^t$ and $\I$ is the interior operation dual to $\C$. Then $\Sigma=\{p_0\}\cup\{a_n: n<\omega\}\sub A_0$.
 
It suffices to show that the set $\Sigma\cup\{-\C^t\{q,-q\}\}$ has the finite meet property in $A_0$: every finite subset has non-zero meet. For then $\Sigma\cup\{-\C^t\{q,-q\}\}$ is included in an ultrafilter $x_0$ of $A_0$ which includes $\Sigma$ but does not contain $\C^t\{q,-q\}$ as it contains $-\C^t\{q,-q\}$.

For each positive integer $m$, let $\Sigma_m=\{p_0\}\cup\{a_n: n<m\}$.  Any finite subset of $\Sigma\cup\{-\C^t\{q,-q\}\}$ is a subset of $\Sigma_m\cup\{-\C^t\{q,-q\}\}$ for some $m$, so  it suffices now to show that 
$\meet(\Sigma_m\cup\{-\C^t\{q,-q\}\})\ne 0$ for any $m$.

Define a quasi-ordered set $(S_m,R_m)$ by $S_m=\{0,\dots,m\}$ and $xR_my$ iff $x\leq y$. Put $p_n'=\{n\}$ for all 
$n\leq m$ and let $q'=\{n\leq m: m\text{ is odd}\}$. Then by the freeness property there exists a tangled closure algebra homomorphism  $f$ from  $(A_0,\C^t)$ to the powerset algebra $(A_{S_m},\C^t_{R_m})$ of $(S_m,R_m)$ such that  $f(p_n)=p'_n$ for all $n\leq m$ and $f(q)=q'$.

For $n<m$, let $a_n'=f(a_n)$. Since $f$ preserves the closure algebra operations, $a_n'$ is the subset of $S_m$ specified  by replacing $p_k$  by $p_k'$ and $q$ by $q'$ in $a_n$. Let 
$\Sigma'_m=\{p'_0\}\cup\{a'_n: n<m\}$.
It is evident that $0\in \bigcap(\Sigma'_m\cup\{-\C^t_{R_m}\{q',-q'\}\})$, since 0 belongs to $p_0'$ and to each $a_n'$, and any endless $R_m$-path is ultimately constant, so cannot move in and out of  $q'$ endlessly, hence 
$0\notin \C^t_{R_m}\{q',-q'\}$.

But if we had $\meet(\Sigma_m\cup\{-\C^t\{q,-q\}\})= 0$ in $A_0$, then as $f$ preserves all the operations involved, we would have
$\bigcap(\Sigma'_m\cup\{-\C^t_{R_m}\{q',-q'\}\})= \emptyset$, a contradiction.

\end{proof}

Now taking the algebra $(A_0,\C^t)$  given by this Lemma,
let $U_0$ be the set of ultrafilters of $A_0$. 
Define a relation $R$ on $U_0$ by putting $xRy$ iff $\{a:\I a\in x\}\sub y$, or equivalently iff $\{\C a:a\in y\}\sub x$. Then it is standard theory that $R$ is a quasi-order  on $U_0$, and has, for all $a\in A_0$ and $x\in U_0$,
\begin{eqnarray}
\label{Ican}
\I a\in x &&\text{iff\quad  for all } y\in U_0, \ xRy \text{ implies }a\in y.
\\
\C a\in x &&\text{iff\quad for some } y\in U_0, \ xRy \text{ and }a \in y.  \label{Ccan}
\end{eqnarray}
Let 
$
U=\{y\in U_0:x_0Ry\},
$
where $x_0$ is the  ultrafilter given by the Lemma. For $a\in A_0$, put 
$$
\ab{a}=\{x\in U:a\in x\}.
$$
Then $A=\{\ab{a}:a\in A_0\}$ is a Boolean subalgebra of the powerset algebra of $U$, since $U\setminus\ab{a}=\ab{{-a}}$ and $\ab{a}\cap\ab{b}=\ab{a\land b}$. The map $a\mapsto \ab{a}$ is a Boolean algebra homomorphism from $A_0$ onto $A$.

We now transfer the tangled closure operation $\C^t$ on $A_0$ to one on $A$, by defining 
\begin{equation} \label{defCAt}
\C_{A}^t\{\ab{\g}:\g\in\G\}=\ab{\C^t\G}
\end{equation}
for all $\G\in\c P_{fin}(A_0)$. We need to check that this is well-defined, i.e.\ that if $\ab{\g}=\ab{\g'}$ for all $\g\in\G$, and
$\G'=\{\g':\g\in\G\}$, then $\ab{\C_A^t\G}=\ab{\C_A^t\G'}$
But  we have  $\ab{\g}=\ab{\g'}$ iff $\g$ and $\g'$ belong to the same members of $U$, which is equivalent to requiring that 
$\g\liff\g'$ belongs to every member of $U$. By   \eqref{Ican} with $x=x_0$, this is equivalent to having $\I(\g\liff\g')\in x_0$. Hence the well-definedness follows because $(A_0,\C^t)$ satisfies
$$
\meet_{\g\in\G}\I(\g\liff\g')\leq\I\big(\C^t\G\liff \C^t\G'\big)
$$
by Theorem \ref{thm:cong} (2), so if $\I(\g\liff\g')\in x_0$ for all $\g\in\G$, then $\I(\C^t\G\liff \C^t\G')\in x_0$.

The unary operation $\C_A$ induced by $\C^t_A$ is given by 
$$
\C_{A}\ab{a}=\C_{A}^t\{\ab{a}\}=\ab{\C^t\{a\}}=\ab{\C a}, 
$$
and its dual has
$\I_{A}\ab{a}=\ab{\I a}$. Equation \eqref{defCAt} ensures that $a\mapsto \ab{a}$ is a homomorphism from $(A_0,\C^t,\C)$  onto $(A,\C^t_A,\C_A)$. Hence $(A,\C^t_A,\C_A)$ is a closure algebra satisfying Fix and Ind, so is a tangled closure algebra.

\begin{theorem}
If $(B,\C_B^t)$ is any tangled closure algebra for which $B$ is complete, then
there is no tangled closure embedding of $(A,\C_A^t)$ into  $(B,\C_B^t)$.
\end{theorem}

\begin{proof}
Assume for the sake of contradiction that there exists  a $f:(A,\C_A^t) \to (B,\C_B^t)$ that is a tangled closure embedding. Then  we show that  
 $f\C_A^t\{\ab{q},\ab{{-q}}\}\}\ne\C_B^t\{f\ab{q},f\ab{{-q}}\}$, which contradicts the assumption that $f$ preserves tangled closure.
 
 By Theorem \ref{thm:cong} (1), $\C\C^t\{q,-q\}=\C^t\{q,-q\}\notin x_0$, so by \eqref{Ccan}, $\C^t\{q,-q\}\notin y$ for all $y\in U$, hence $\ab{\C^t\{q,-q\}}=\emptyset$.
 Thus $f\C_A^t\{\ab{q},\ab{{-q}}\}=f\ab{\C^t\{q,-q\}}=f\emptyset=0$ in $B$. Therefore to prove that $f$ is not a tangled closure homomorphism it suffices to show that  $\C_B^t\{f\ab{q},f\ab{{-q}}\}\ne 0$.
 
To show this, put $b_n=f\ab{p_n}$  for each $n<\omega$, and let $b=\join\{b_n:n<\omega\}$. Then $b$ exists in $B$ as $B$ is complete. We prove that $b$ is a post-fixed point for $\{f\ab{q},f\ab{{-q}}\}$, i.e.\
\begin{equation} \label{fixpointb}
b\leq   \C_B(b\land f\ab{{q}})\land \C_B(b\land f\ab{{-q}}).
\end{equation}

Now  if $n$ is even, then  since $a_n\in x_0$ it follows  by \eqref{Ican} that
$p_{n}\Rightarrow \C(p_{n+1}\land q)\in y$ for all $y\in U$, hence
$
\ab{p_n}\sub\ab{\C(p_{n+1}\land q)}=\C_A(\ab{p_{n+1}}\cap \ab{q}).
$
Similarly, if $n$ is odd, then 
$
\ab{p_n}\sub\C_A(\ab{p_{n+1}}\cap \ab{{-q}}).
$
Since $f$ is a closure algebra homomorphism, this implies that for all $n<\omega$,
\begin{align}
&b_n\leq\C_B(b_{n+1}\land f\ab{q}), \quad \text{ if $n$ is even;}   \label{even}
\\
&b_n\leq\C_B(b_{n+1}\land f\ab{{-q}}), \quad \text{ if $n$ is odd}.  \label{odd}
\end{align}
Thus if $n$ is even, then by \eqref{even} $b_n\leq\C_B(b_{n+1}\land f\ab{q}))\leq\C_B(b\land f\ab{q}))$. Also then as $n+1$ is odd we use \eqref{odd} with $n+1$ in place of $n$ to infer that 
$$
\C_B(b_{n+1}) \leq \C_B\C_B(b_{n+2}\land f\ab{{-q}})\leq \C_B(b\land f\ab{{-q}}).
$$
Since $b_n\leq\C_B(b_{n+1})$ follows from \eqref{even}, altogether these facts  imply that
\begin{equation} \label{fixpointbn}
b_n\leq   \C_B(b\land f\ab{{q}})\land \C_B(b\land f\ab{{-q}})
\end{equation}
when $n$ is even. But a similar proof shows that \eqref{fixpointbn} also holds when $n$ is odd. Hence it holds for all $n<\omega$, from which \eqref{fixpointb} follows.

Thus $b$ is indeed a post-fixed point for $\{f\ab{q},f\ab{{-q}}\}$, so $b\leq \C_B^t\{f\ab{q},f\ab{{-q}}\} $ by Lemma \ref{lem:defCt}. But as $p_0\in x_0$ we have $x_0\in\ab{p_0}$, hence $\ab{p_0}\ne \emptyset$. 
So as $f$ is injective,
$$
0=f\emptyset\ne f\ab{p_0}=b_0\leq b\leq \C_B^t\{f\ab{q},f\ab{{-q}}\}.
$$
This proves $ \C_B^t\{f\ab{q},f\ab{{-q}}\}\ne 0$, which completes the proof as explained.
\end{proof}

Thus $(A,\C_A^t)$ has no homomorphic embedding into any complete tangled closure algebra. In particular it is not embeddable into the algebra $(A_S,\C^t_S)$ of subsets of any topological space $S$, including not being embeddable into 
the algebra $(A_S,\C^t_R)$ of subsets of any quasi-ordered set  $(S,R)$.

Let $B$ be any complete extension of the Boolean algebra $A$; take $\C_B$ to be the closure operator on $B$ extending $\C_A$ defined by \eqref{complC}; and let $\C^t_B$ be the expansion of $\C_B$ given by \eqref{defCt}.
Then  $ (B,\C_B^t)$ is a tangled closure algebra by Theorem \ref{thm:compcla}.
The inclusion $A\hookrightarrow B$ provides the promised example of a map 
$(A,\C_A^t)\to (B,\C_B^t)$ that is a homomorphism of the associated closure algebra reducts but is not a tangled closure homomorphism (by Theorem \ref{findom}, such an example must have infinite $A$). It also provides the promised example in which $(A,\C_A)$ is a subalgebra of $(B,\C_B)$ while $(A,\C_A^t)$ is not a subalgebra of $(B,\C_B^t)$. 

Finally we note that this absence of complete extensions is not attributable to the fact that tangled closure algebras can be seen as having infinite signature (Remark \ref{infsig}). The construction needs just
the element $\C_A^t\{\ab{q},\ab{{-q}}\}$ which requires only  the binary operation $(a,b)\mapsto\C_A^t\{a,b\}$ of this signature for its formation.

\bibliographystyle{plain}

\bibliographystyle{plain}

\begin{thebibliography}{10}

\bibitem{bana:hull56}
B.~Banaschewski.
\newblock H{\"u}llensysteme und {E}rweiterung von {Q}uasi-{O}rdnungen.
\newblock {\em Zeitschrift f{\"u}r Mathematische Logik und Grundlagen der
  Mathematik}, 2:117--130, 1956.

\bibitem{bana:cate67}
B.~Banaschewski and G.~Bruns.
\newblock Categorical characterization of the {M}ac{N}eille completion.
\newblock {\em Archiv der Mathematik}, 18:369--377, 1967.

\bibitem{dave:intr90}
B.~A. Davey and H.~A. Priestley.
\newblock {\em Introduction to Lattices and Order}.
\newblock Cambridge University Press, 1990.

\bibitem{dawa:moda09}
Anuj Dawar and Martin Otto.
\newblock Modal characterisation theorems over special classes of frames.
\newblock {\em Annals of Pure and Applied Logic}, 161:1--42, 2009.

\bibitem{dumm:moda59}
M.~A.~E. Dummett and E.~J. Lemmon.
\newblock Modal logics between {S}4 and {S}5.
\newblock {\em Zeitschrift f{\"u}r Mathematische Logik und Grundlagen der
  Mathematik}, 5:250--264, 1959.

\bibitem{fern:tang11}
David Fern\'{a}ndez-Duque.
\newblock Tangled modal logic for spatial reasoning.
\newblock In Toby Walsh, editor, {\em Proceedings of the Twenty-Second
  International Joint Conference on Artificial Intelligence (IJCAI)}, pages
  857--862. AAAI Press/IJCAI, 2011.

\bibitem{fern:tang12}
David Fern{\'{a}}ndez-Duque.
\newblock Tangled modal logic for topological dynamics.
\newblock {\em Annals of Pure and Applied Logic}, 163:467--481, 2012.

\bibitem{giva:intr09}
Steven Givant and Paul Halmos.
\newblock {\em Introduction to {B}oolean Algebras}.
\newblock Springer, 2009.

\bibitem{gold:fini16}
Robert Goldblatt and Ian Hodkinson.
\newblock The finite model property for logics with the tangle modality.
\newblock Submitted.

\bibitem{gold:spat16}
Robert Goldblatt and Ian Hodkinson.
\newblock Spatial logic of tangled closure operators and modal mu-calculus.
\newblock Submitted.

\bibitem{gold:spat14}
Robert Goldblatt and Ian Hodkinson.
\newblock Spatial logic of modal mu-calculus and tangled closure operators.
\newblock \url{arxiv.org/abs/1603.01766}, 2016.

\bibitem{gold:tang16}
Robert Goldblatt and Ian Hodkinson.
\newblock The tangled derivative logic of the real line and zero-dimensional
  spaces.
\newblock In Lev Beklemishev, St\'{e}phane Demri, and Andr\'{a}s M\'{a}t\'{e},
  editors, {\em Advances in Modal Logic, Volume 11}, pages 342--361. College
  Publications, 2016.

\bibitem{john:elem01}
Peter Johnstone.
\newblock Elements of the history of locale theory.
\newblock In C.~E. Aull and R.~Lowen, editors, {\em Handbook of the History of
  General Topology}, volume~3, pages 835--851. Kluwer Academic Publishers,
  2001.

\bibitem{jons:bool51}
Bjarni J{\'o}nsson and Alfred Tarski.
\newblock {B}oolean algebras with operators, part {I}.
\newblock {\em American Journal of Mathematics}, 73:891--939, 1951.

\bibitem{kura:oper22}
Casimir Kuratowski.
\newblock Sur l'op{\'e}ration $\bar{A}$ de l'{A}nalysis {S}itus.
\newblock {\em Fundamenta Mathematicae}, 3:182--199, 1922.

\bibitem{macn:part37}
H.~M. MacNeille.
\newblock Partially ordered sets.
\newblock {\em Transactions of the American Mathematical Society}, 42:416--460,
  1937.

\bibitem{mcki:alge44}
J.~C.~C. McKinsey and Alfred Tarski.
\newblock The algebra of topology.
\newblock {\em Annals of Mathematics}, 45:141--191, 1944.

\bibitem{mcki:clos46}
J.~C.~C. McKinsey and Alfred Tarski.
\newblock On closed elements in closure algebras.
\newblock {\em Annals of Mathematics}, 47:122--162, 1946.

\bibitem{monk:comp70}
J.~D. Monk.
\newblock Completions of {B}oolean algebras with operators.
\newblock {\em Mathematische Nachrichten}, 46:47--55, 1970.

\bibitem{rasi:alge51}
H.~Rasiowa.
\newblock Algebraic treatment of the functional calculi of {H}eyting and
  {L}ewis.
\newblock {\em Fundamenta Mathematicae}, 38:99--126, 1951.

\bibitem{rasi:math63}
Helena Rasiowa and Roman Sikorski.
\newblock {\em The Mathematics of Metamathematics}.
\newblock PWN--Polish Scientific Publishers, Warsaw, 1963.

\bibitem{siko:bool64}
R.~Sikorski.
\newblock {\em Boolean Algebras}.
\newblock Springer-Verlag, Berlin, second edition, 1964.

\bibitem{tars:auss38}
Alfred Tarski.
\newblock Der {A}ussagenkalk{\"u}l und die {T}opologie.
\newblock {\em Fundamenta Mathematicae}, 31:103--134, 1938.
\newblock English translation by J. H. Woodger as \textit{Sentential Calculus
  and Topology} in \cite{tars:logi56}, 421--454.

\bibitem{tars:logi56}
Alfred Tarski.
\newblock {\em Logic, Semantics, Metamathematics: Papers from 1923 to 1938}.
\newblock Oxford University Press, 1956.
\newblock Translated into English and edited by J. H. Woodger.

\bibitem{theu:macn07}
Mark Theunissen and Yde Venema.
\newblock Mac{N}eille completions of lattice expansions.
\newblock {\em Algebra Universalis}, 57:143--193, 2007.

\bibitem{bent:moda76corr}
J.~F. A.~K. van Benthem.
\newblock {\em Modal Correspondence Theory}.
\newblock PhD thesis, University of Amsterdam, 1976.

\bibitem{bent:moda83}
J.~F. A.~K. van Benthem.
\newblock {\em Modal Logic and Classical Logic}.
\newblock Bibliopolis, Naples, 1983.

\end{thebibliography}

\end{document}